\documentclass[reqno]{amsart}
\usepackage{amsfonts,amsmath,amsthm,amssymb,mathrsfs}
\usepackage{color}
\usepackage{amsmath,amssymb,amsfonts,bm}
\usepackage{graphicx}
\usepackage{CJKutf8}
\usepackage{newunicodechar}
\usepackage{xcolor}
\numberwithin{equation}{section}
\usepackage[pdfstartview=FitH,colorlinks=true]{hyperref}
\hypersetup{urlcolor=red, linkcolor=blue,citecolor=red, CJKbookmarks=true}
\allowdisplaybreaks

\newtheorem{theorem}{Theorem}[section]

\newtheorem{lemma}[theorem]{Lemma}

\theoremstyle{definition}

\newtheorem{remark}{Remark}[section]

\newcommand{\abs}[1]{\left\vert#1\right\vert}

\newcommand{\norm}[1]{\left\Vert#1\right\Vert}

\newcommand{\comii}[1]{\left<#1\right>}

\makeatletter

\newcommand{\Rmnum}[1]{\expandafter\@slowromancap\romannumeral #1@}
\makeatother

\title[Weighted regularity for Euler Equation]
{Weighted Gevrey class regularity of Euler equation in the whole space}
\author{Feng Cheng, Wei-Xi Li and Chao-Jiang Xu}
\date{}

\address{\noindent \textsc{Feng Cheng, School of Mathematics and Statistics, Wuhan university 430072, Wuhan, P.R. China}}
\email{chengfengwhu@whu.edu.cn}

\address{\noindent \textsc{Wei-Xi Li, School of Mathematics and Statistics,    and Computational Science Hubei Key Laboratory, Wuhan university 430072, Wuhan, P.R. China}}
\email{wei-xi.li@whu.edu.cn}

\address{\noindent \textsc{Chao-Jiang Xu, Universit\'e de Rouen, CNRS UMR 6085, Laboratoire de Math\'ematiques, 76801 Saint-Etienne du Rouvray, France\\
and\\
School of Mathematics and Statistics, Wuhan university 430072, Wuhan, P.R. China}}
\email{Chao-Jiang.Xu@univ-rouen.fr}
\begin{document}


\begin{abstract}
In this paper we study the weighted Gevrey class regularity of Euler equation in the whole space $\mathbb{R}^3$. We first establish the local existence of Euler equation in weighted Sobolev space, then obtain the weighted Gevrey regularity of Euler equation. We will use the weighted Sobolev-Gevrey space method to obtain the results of Gevrey regularity of Euler equation, and the use of the property of singular operator in the estimate of the pressure term is the improvement of our work.
\end{abstract}

\maketitle

\section{Introduction}
The incompressible flow for Euler equations in the whole space $\mathbb{R}^3$ reads
\begin{equation}\label{1.1}
\left\{
\begin{aligned}
 &\frac{\partial u}{\partial t}+u\cdot\nabla u+\nabla p=0,\,x\in \mathbb{R}^3,\,t>0 \\
 &\nabla\cdot u=0,\,x\in\mathbb{R}^3,\,t>0\\
 &u|_{\abs{x}\to\infty}=0,\,t>0\\
 &u|_{t=0}=u_0,\,x\in\mathbb{R}^3,
\end{aligned}
\right.
\end{equation}
where $u(x,t)=(u_1(x,t), u_2(x,t), u_3(x,t))$ denotes vector velocity field and $p(x,t)$ denotes scalar pressure at point $x=(x_1, x_2, x_3)$ at time $t$. Usually, we always assume the initial data $u_0$ satisfies the following compatible condition
\begin{equation}\label{1.2}
 \nabla\cdot u_0=0,\,\, u_0|_{\abs{x}\to\infty}=0.
\end{equation}
There are many known results about Euler equations in history. It is well known that in two-dimensional space, the existence and uniqueness of (globally in time) classical solutions to the Euler equations were studied in \cite{MB, M, S, Y} in either spaces of continuous functions or H\" older functions. With the use of the Log-Sobolev inequality, one can also obtain the global existence in Sobolev space for reference in \cite{OM}.  While in three-dimensional spaces, there are only local existence and uniqueness of $H^r$- solutions, with $r>3/2+1$, on a maximal time interval $[0,T^\ast)$, for references in \cite{BB1,EM,K,S1,T}. The famous BKM criterion \cite{BK} assures that the solutions exist on $[0,T]$ as long as $\int_0^T \norm{{\text{curl}}\ u(s)}_{L^\infty}ds$ is bounded, and this is the reason that in two dimensions the solutions exist globally. For $C^\infty$ smooth initial data, Foias, Frisch and Temam \cite{FFT} proved the persistence of $C^\infty$ solutions. For analytical initial data, Bardos and Benachour \cite{BB} proved the persistence of analyticity of solutions to Euler equations and also obtain an estimate of the radius of analyticity in \cite{B,BS}. Since Gevrey class space is the intermediate space between $C^\infty$ smooth space and analytical function space, it is natural to consider such results in the framework of Gevrey space. On three-dimensional periodic domains, Levermore and Oliver \cite{LO} take the method of Gevrey-class regularity to prove the persistence of analyticity and obtain the explicit estimate of the decay of the radius. Their method is based on the fact that the Gevrey class space can be equivalently defined by characterizing the decay of their Fourier coefficients in periodic domain, see \cite{FT}. Later, Kukavica and Vicol developed this method to half space and improved the estimate of the decay of the radius of analyticity (or Gevrey class regularity) in their work \cite{KV, KV1}.

We have discussed the persistence of vertical weighted Gevrey class regularity of Euler equation on half plane $\mathbb{R}^2_+$ in \cite{CLX} following the method of Kukavica and Vicol \cite{KV}, where the difficulty arise from the estimate of the pressure in weighted Sobolev space. In this paper, we will discuss the whole spatial variable weighted Gevrey class regularity of Euler equation in whole space. We first consider the local existence and uniqueness of solutions for Euler equation in the weighted Sobolev space. Then we consider the persistence of weighted Gevrey class regularity of solutions of Euler equation. The appearance of the weight function and boundary condition will cause trouble in estimate of the pressure, thus we will now consider in whole space without boundary and implement the theory of singular integral operators with the weight function belonging to certain weight class $\mathcal{A}_p$. In the future, we will look forward to investigate situations within bounded domain. We remark that the results can also be applied to two dimensions which recover the results we have studied in \cite{CLX}.

The paper is organized as follows. In Section 2, we will give some notations and state our main results. In Section 3, we study the local existence of solutions of Euler equation in Weighted Sobolev space. In Section 4, we will give two Lemmas first and then use the Lemmas and Theorem 3.1 to finish the proof of Theorem 2.1.

\section{Preliminaries}

In this section we will give some notations and function spaces which will be used throughout the following arguments. Throughout the paper, $C$ denotes a generic constant which may vary from line to line.

For a multi-index $\alpha=(\alpha_1,\alpha_2,\alpha_3)$ in $\mathbb{N}_0^3$, we denote $\abs{\alpha}=\alpha_1+\alpha_2+\alpha_3$ and $\partial^\alpha=\partial_{x_1}^{\alpha_1}\partial_{x_2}^{\alpha_2}\partial_{x_3}^{\alpha_3}$. We denote by $L^2(\mathbb{R}^3)^3$ the space of real valued vector functions which are square integrable, and it is a Hilbert space for the scalar product
$$
 \comii{u,v}=\int_{\mathbb{R}^3}u(x)\cdot v(x)dx,\quad \norm{u}_{L^2}^2:=\comii{u,u}.
$$
With no ambiguity arise we may suppress the domain $\mathbb{R}^3$ and the differences between vector functions and scalar functions and denote uniformly by $L^2$ for simplicity.
Likewise, we denote by $H^r$ the standard Sobolev space of vector functions which are in $L^2$ together with their weak derivatives of order $\leq r$, and the inner product and the norm are defined as follows
$$
  \comii{u,v}_{H^r}=\sum_{\abs{\alpha}\leq r}\comii{\partial^\alpha u,\partial^\alpha v},\quad\norm{u}_{H^r}^2:=\sum_{\abs{\alpha}\leq r}\norm{\partial^\alpha u}_{L^2}^2.
$$
Let us define the weight function $\comii{x}:=(1+\abs{x}^2)^{1\over 2}$, which is very close to $\abs{x}$ when $\abs{x}$ is very large and is well behaved when $x$ is near to zero.
We then introduce the weighted Sobolev space $H^m_\ell$ as follows,
$$
 H^r_\ell=\bigg\{u\in H^r\ ;\ \norm{u}_{H^r_\ell}^2:=\norm{u}_{L^2}^2+\sum_{1\leq\abs{\alpha}\leq r}\norm{\comii{x}^\ell\partial^\alpha u}_{L^2}^2<\infty\bigg\},
$$
where
$$
\norm{\comii{x}^\ell\partial^\alpha u}_{L^2}^2=\int_{\mathbb{R}^3}\comii{x}^{2\ell} \big|\partial^\alpha u(x)\big|^2 dx=\int_{\mathbb{R}^3}(1+\abs{x}^2)^\ell \big|\partial^\alpha u(x)\big|^2 dx,
$$
and $\ell\geq 0$ is a constant. It is obvious that the weighted Sobolev space $H^r_\ell$ equipped with the following inner product,
$$
 \comii{u,v}_{H^r_\ell}:=\comii{u,v}+\sum_{1\leq\abs{\alpha}\leq r}\comii{\comii{x}^\ell \partial^\alpha u,\comii{x}^\ell \partial^\alpha v},
$$
is a Hilbert space. And $\ell\geq0$ obviously imply that for a given function $u\in H^r_\ell$, $\norm{u}_{H^r}\leq\norm{u}_{H^r_\ell}$ always holds.

In this paper we will consider the Gevrey regularity for Euler equation. Let us recall the definition of Gevrey class functions first. It is said that a smooth function $u(x)$ is uniformly of Gevrey class s in $\mathbb{R}^3$, if there exists $C,\tau>0$ such that
\begin{equation}\label{2.1}
 \abs{\partial^\alpha u(x)}\leq C{{\abs{\alpha}!^s}\over{\tau^{\abs{\alpha}}}},
\end{equation}
for all $x\in\mathbb{R}^3$ and all multi-index $\alpha\in\mathbb{N}_0^3$.
When $s=1$ these functions are of the class of real-analytic functions, and for $s>1$ these functions are $C^\infty$ smooth but might not be analytic. The constant $\tau$ is called the radius of analyticity with $s=1$ (or Gevrey class regularity with $s>1$ respectively). Usually we refer to the equivalently defined space somewhere called Sobolev-Gevrey spaces, for example the $X_\tau, Y_\tau$ used in Kukavica and Vicol \cite{KT} and \cite{KV},
\begin{equation*}
 X_\tau =\left\{v\in C^\infty\ ;\ \norm{v}_{X_\tau}=\sum_{m=3}^\infty \abs{v}_{m} \frac{\tau^{m-3}}{(m-3)!^s}<\infty \right\},
\end{equation*}
where $\abs{v}_m$ is defined as
\begin{equation*}
 \abs{v}_m=\sum_{\abs{\alpha}=m}\norm{\partial^\alpha v}_{L^2},
\end{equation*}
and usually we define $\abs{v}_{m,\infty}=\sum_{\abs{\alpha}=m}\norm{\partial^\alpha v}_{L^\infty}$.
Similarly, $Y_\tau$ is defined as
\begin{equation*}
 Y_\tau =\left\{v\in C^\infty\ ;\ \norm{v}_{Y_\tau}=\sum_{m=4}^\infty \abs{v}_m \frac{(m-3)\tau^{m-4}}{(m-3)!^s}<\infty \right\}.
\end{equation*}
It is remarked that the space $X_\tau$ and $Y_\tau$ can be identified with the the classical definition \eqref{2.1}.

Naturally, we say that a smooth function $u(x)$ is uniformly of weighted Gevrey class s, if there exists $C,\tau>0$ such that
\begin{equation}\label{2.5}
 \abs{\comii{x}^\ell \partial^\alpha u(x)}\leq C{{\abs{\alpha}!^s}\over{\tau^{\abs{\alpha}}}},
\end{equation}
for all $x\in\mathbb{R}^3$ and all multi-index $\alpha\in\mathbb{N}_0^3$. It is obvious that a weighted Gevrey class s function of course belongs to the standard Gevrey class s for $\ell\geq0$. In analogy with $X_\tau$ and $Y_\tau$, we define the weighted Sobolev-Gevrey class s space $X_{\tau,\ell}$ and $Y_{\tau,\ell}$ as follows,
\begin{equation*}
 X_{\tau,\ell}=\left\{v\in C^\infty\ ;\ \norm{v}_{X_{\tau,\ell}}=\sum_{m=3}^\infty \abs{v}_{m,\ell}\frac{\tau^{m-3}}{(m-3)!^s}\right\},
\end{equation*}
where
$$\abs{v}_{m,\ell}:=\sum_{\abs{\alpha}=m}\norm{\comii{x}^\ell \partial^\alpha v}_{L^2},\,\abs{ v}_{m,\ell,\infty}:=\sum_{\abs{\alpha}=m}\norm{\comii{x}^\ell \partial^\alpha v}_{L^\infty},$$
and
\begin{equation*}
 Y_{\tau,\ell}=\left\{v\in C^\infty\ ;\ \norm{v}_{Y_{\tau,\ell}}=\sum_{m=4}^\infty \abs{ v}_{m,\ell}\frac{(m-3)\tau^{m-3}}{(m-3)!^s}\right\}.
\end{equation*}
\begin{remark}
The functions space $X_{\tau,\ell}$ and $Y_{\tau,\ell}$ can be identified with \eqref{2.5}. Namely, a function $u$ is said to be of weighted Gevrey class s with radius of weighted Gevrey class $\tau$, then $u\in X_{\tau^\prime,\ell}$ for $\tau^\prime<\tau$. Conversely, if $u\in X_{\tau,\ell}$, then $u$ is of course of weighted Gevrey class s with $\tau$ the radius of weighted Gevrey class regularity.
\end{remark}

The following statement is our main theorem concerning the persistence of weighted Gevrey class regularity for the Euler equation in whole space.

\begin{theorem}\label{Th2.1}
Let $r\geq 5, 0\leq\ell<\frac{3}{2}, s\geq1, \tau_0>0$ be given constants. Let $u_0\in H^r_\ell\cap X_{\tau_0+\epsilon,\ell}$ be divergence-free and uniformly of weighted Gevrey-class s in $\mathbb{R}^3$, where $\epsilon<\tau_0$. Then we have,\\
(1). There exists unique solutions $u,p$ such that
\begin{equation*}
u\in L^\infty([0,T^\ast),H^r_\ell),\ \nabla p\in L^\infty([0,T^\ast), H^r_\ell),
\end{equation*}
where $0<T^\ast<\infty$ is the maximal time of existence of $H^r-$ solution.\\
(2). The solutions $u,\nabla p$ are also of weighted Gevrey-class s and satisfying,
$$
u(t,x)\in L^\infty([0,T);\ X_{\tau(t),\ell})\cap L^1([0,T);\ Y_{\tau(t),\ell}),
$$
and
$$
\nabla p(t,x)\in L^\infty([0,T);\ X_{\tau(t),\ell})\cap L^1([0,T);\ Y_{\tau(t),\ell}),
$$
where $0<T<T_\ast$.
Moreover, the uniform radius of weighted Gevrey-class $\tau(t)$ of $u(t)$ is a decreasing function of $t$ and satisfying
\begin{equation}\label{2.9}
\tau(t)\geq \frac{(1-C\norm{u_0}_{H^r}t)^2}{C_0(1+t)^4},
\end{equation}
where $C_2>0$ is a constant depending only on $r$, while $C_1$ has additional dependence on $u_0$.
\end{theorem}

\begin{remark}
We also obtain an explicitly estimate of the decay of the radius of Gevrey class s, but we did not actually improve the
rate of decay \eqref{2.9} as to \cite{KV} which was proportional to $\exp\big(\int_0^t \norm{\nabla u(s)}_{L^\infty}ds\big)$. The reason is that the calculus inequality in \cite{MB} failed in appearance of the weight function, thus we can only obtain an rougher estimate.
\end{remark}

\begin{remark}
The proof of Theorem \ref{Th2.1} also works in two dimensional plane, in which case the power $\ell$ need to be less than $1$ due to the property of $\mathcal{A}_2$ weights for Rieze operator. It has been done in \cite{CLX} in half plane with vertical variable weight function. And it is well known that in two-dimensional case the $H^m$-solution exists globally in time.
\end{remark}

The rest of the paper is due to proving Theorem \ref{Th2.1}. We will prove the first part of Theorem \ref{Th2.1} in Section \ref{Sec3}, and the weighted Gevrey class regularity will be postponed to Section \ref{Sec4}.

\section{Local solution to Euler equation in weighted Sobolev space}\label{Sec3}
In this Section we will prove the first part $(1)$ of the main Theorem \ref{Th2.1}, i.e. the local existence and uniqueness of solutions to Euler equation in weighted Sobolev space. When there is no weight function $\comii{x}^\ell$, the existence and uniqueness is classical in \cite{K,MB}, which we state as follows,
\begin{theorem}
Fix $r\geq3$. Let $u_0\in H^r$ be divergence-free. Then there exists a unique solution $u\in C([0,T_\ast),H^r)$ and $\nabla p\in C([0,T_\ast),H^r)$ to the Euler equation \eqref{1.1}, where $T_\ast$ is the maximal existing time of $H^r(\mathbb{R}^3)-$ solution to the Euler equation \eqref{1.1}.
\end{theorem}

When the weight function was taken in consideration, the situation becomes a slight differently cause the pressure term does not vanish. In this situation the inner product with weight is no longer orthogonal between a gradient function and a divergence free vector field, which is the main difficulty when considering in framework of the weighted function space. To overcome the difficulty arise from the estimate of the pressure, we take use of the tool of singular integral theory, which was recently used in \cite{YZ} in estimate of the pressure term when studying the global well-posedness of incompressible magnetohydrodynamic equations.

\begin{remark}
In two-dimensional plane, it is shown in \cite{OM} the $H^r$-solution exists globally in time and it can be proved that the $H^r_\ell$-solution also exists globally in time with $0\leq\ell\leq1$.  The reason $\ell$ less than 1 is due to the $\mathcal{A}_2$ weights theory of Rieze operator in $\mathbb{R}^2$.
\end{remark}

In the following, we will first give some Lemmas and then we give the detailed proof of the local existence and uniqueness, i.e. the first part $(1)$ of Theorem \ref{Th2.1}. Firstly, we will frequently use the following embedding inequality in the weighted Sobolev space $H^m_\ell$, which is as follows,
\begin{lemma}
For any $\ell\geq0$, there holds for $u\in H^3_\ell$
$$
 \norm{\comii{x}^\ell \nabla u}_{L^\infty} \leq C\norm{u}_{H^3_\ell},
$$
where $C$ is a constant independent of $u$.
\end{lemma}
This is an elementary Sobolev embedding theorem $H^2\hookrightarrow L^\infty$, and using the fact that $\partial^\alpha \comii{x}^\ell\leq C\comii{x}^\ell$ holds for $\forall \alpha\in\mathbb{N}^3$, we thus omit the proof. In order to handle the convecting term, we need the following weighted calculus inequality.

\begin{lemma}\label{lem3.3}
For $\alpha,\beta\in\mathbb{N}^3$ with $0\neq\beta\leq\alpha$ and $\abs{\alpha}\leq m$, let $u\in H^r_\ell$ with $r\geq 3,\ell\geq0$, the following inequality holds,
\begin{equation*}
 \norm{\comii{x}^\ell \partial^\beta u\cdot\nabla\partial^{\alpha-\beta}u}_{L^2} \leq C\norm{u}_{H^r}\norm{u}_{H^r_\ell},
\end{equation*}
where $C$ is a constant depending on $r$.
\end{lemma}
\begin{proof}
The proof is the application of Sobolev embedding inequality. It suffices to consider the case $\abs{\alpha}=r$, since the case for $\abs{\alpha}<r$ is much easier, for example, in such case we directly have for some constant $C$,
\begin{equation*}
\begin{split}
 &\norm{\comii{x}^\ell \partial^\beta u\cdot\nabla \partial^{\alpha-\beta}u}_{L^2}\\
 &\leq \norm{\comii{x}^\ell \partial^\beta u}_{L^4}\norm{\nabla \partial^{\alpha-\beta}u}_{L^4} \\
 &\leq C\norm{\comii{x}^\ell \partial^\beta u}_{L^2}^{1/4}\norm{D(\comii{x}^\ell \partial^\beta u)}_{L^2}^{3/4}\norm{\nabla \partial^{\alpha-\beta}u}_{L^2}^{1/4}\norm{D\nabla \partial^{\alpha-\beta}u}_{L^2}^{3/4}\\
 &\leq C\norm{u}_{H^r_\ell}\norm{u}_{H^r},
\end{split}
\end{equation*}
where we have used the Gagliardo-Nirenberg inequality (or Ladyzhenskaya's inequality) for $N=3$ in the following form,
\begin{equation*}
 \norm{u}_{L^4}\leq C\norm{u}_{L^2}^{1/4}\norm{D u}_{L^2}^{3/4}.
\end{equation*}
We now only consider the case $\abs{\alpha}=r$ and discuss the values of $\abs{\beta}$ for $\beta\neq0$. Let us assume first that $r=3$, in such case $\abs{\beta}$ has only three possible values.\\
If $\abs{\beta}=1$, the H\"older inequality and Sobolev inequality imply,
\begin{equation*}
 \norm{\comii{x}^\ell \partial^\beta u\cdot\nabla\partial^{\alpha-\beta}u}_{L^2}\leq \norm{\partial^\beta u}_{L^\infty}\norm{\comii{x}^\ell\nabla\partial^{\alpha-\beta}u}_{L^2}\leq C\norm{u}_{H^r}\norm{u}_{H^r_\ell}.
\end{equation*}
If $\abs{\beta}=2$, the H\"older inequality and the Gagliardo-Nirenberg inequality imply,
\begin{equation*}
\begin{aligned}
   &~~\norm{\comii{x}^\ell \partial^\beta u\cdot\nabla\partial^{\alpha-\beta}u}_{L^2}\\
   &\leq \norm{\comii{x}^\ell \partial^\beta u}_{L^4}\norm{\nabla\partial^{\alpha-\beta}u}_{L^4}\\
   &\leq C\norm{\comii{x}^\ell \partial^\beta u}_{L^2}^{1/4}\norm{D(\comii{x}^\ell \partial^\beta u)}_{L^2}^{3/4}\norm{\nabla\partial^{\alpha-\beta}u}_{L^2}^{1/4}\norm{D(\nabla\partial^{\alpha-\beta}u)}_{L^2}^{3/4}\\
   &\leq C\norm{u}_{H^r}\norm{u}_{H^r_\ell}.
\end{aligned}
\end{equation*}
If $\abs{\beta}=3$, then $\alpha=\beta$ and we have
\begin{equation*}
 \norm{\comii{x}^\ell\partial^\beta u\cdot\partial^{\alpha-\beta}\nabla u}_{L^2}\leq \norm{\nabla u}_{L^\infty}\norm{\comii{x}^\ell\partial^\beta u}_{L^2}\leq C\norm{u}_{H^r}\norm{u}_{H^r_\ell}.
\end{equation*}
For higher r say $r\geq 4$, the situation becomes more easier. In this situation, we only need to discuss two situations, which are $1\leq\abs{\beta}\leq r-2$ and $r-1\leq\abs{\beta}\leq r$.\\
If $1\leq\abs{\beta}\leq r-2$, we have
\begin{equation*}
 \norm{\comii{x}^\ell \partial^\beta u\cdot\nabla\partial^{\alpha-\beta}u}_{L^2}\leq \norm{\partial^\beta u}_{L^\infty}\norm{u}_{H^r_\ell}\leq C\norm{u}_{H^r}\norm{u}_{H^r_\ell}
\end{equation*}
If $r-1\leq\abs{\beta}\leq r$, we have
\begin{equation*}
 \norm{\comii{x}^\ell \partial^\beta u\cdot\nabla\partial^{\alpha-\beta}u}_{L^2}\leq \norm{\comii{x}^\ell\partial^\beta u}_{L^2}\norm{\nabla\partial^{\alpha-\beta}u}_{L^\infty}\leq C\norm{u}_{H^r}\norm{u}_{H^r_\ell}
\end{equation*}
With all the above in consideration, the Lemma is then proved.
\end{proof}

The weighted calculus inequality in Lemma \ref{lem3.3} is rougher than the estimate in the case without weight function in \cite{MB}, and this is the reason the decay of the radius in \eqref{2.9} is rougher than the decay obtained by Kukavica and Vicol \cite{KV}. We note that the pressure term does not vanish because of the weight function, this is the main difficulty to handle the Euler equation in weighted Sobolev space. We implement the idea of Singular operator theory to handle this term. We first recall the well known Weighted Calder\'on-Zygmund inequalities in the framework of singular integrals (see \cite{SEM}). The following inequality holds in $f\in C_0^\infty(\mathbb{R}^N\backslash\{0\})$,
\begin{equation}\label{3.9}
 \norm{\abs{x}^\alpha Tf}_{L^p}\leq C\norm{\abs{x}^\alpha f}_{L^p},
\end{equation}
for $1<p<\infty, {-N\over p}<\alpha<{N\over p^\prime}$, where $T$ is the Calder\'on-Zygmund kernel corresponding to the operator $D^2\Delta^{-1}$. The weight function $\comii{x}^\ell$ is close to $\abs{x}^\ell$ for large $x$, so the above inequality also holds for $\comii{x}^\ell$ if $\ell<{N\over p^\prime}$. Taking $p=2, N=3, f\in C_0^\infty(\mathbb{R}^3)$, the inequality \eqref{3.9} implies that,
\begin{equation}\label{3.10}
 \norm{\comii{x}^\ell D^2 f}_{L^2} \leq C\norm{\comii{x}^\ell \Delta f}_{L^2},
\end{equation}
for $0\leq\ell<{3\over 2}$. With this inequality, we can prove the following Lemma.

\begin{lemma}\label{lem3.4}
Let $r\geq3$ and suppose $p$ satisfies
\begin{equation*}
 -\Delta p=\sum_{i,j=1}^3\partial_{x_i}u_j \partial_{x_j}u_i.
\end{equation*}
Then for fixed multi-index $\alpha\in\mathbb{N}^3$ with $1\leq\abs{\alpha}\leq r$, the following estimate holds,
\begin{equation*}
 \norm{\comii{x}^\ell \nabla\partial^\alpha p}_{L^2} \leq C\norm{u}_{H^r}\norm{u}_{H^r_\ell},
\end{equation*}
where $C$ is a constant depending only on $r$.
\end{lemma}
\begin{proof}
Fix $\alpha$ with $1\leq\abs{\alpha}\leq r$, there exists $\alpha^\prime$ such that $\alpha^\prime\leq\alpha, \abs{\alpha^\prime}=\abs{\alpha}-1$. Since $p$ satisfies
\begin{equation*}
 -\Delta p=\sum_{i,j=1}^3\partial_{x_i}u_j\partial_{x_j}u_i.
\end{equation*}
We rewrite $\nabla \partial^\alpha p=\nabla \partial^{\alpha-\alpha^\prime}\partial^{\alpha^\prime} p$ with $\abs{\alpha-\alpha^\prime}=1$, and thus $\partial^{\alpha^\prime}p$ satisfies
$$
 -\Delta \partial^{\alpha^\prime} p=\sum_{i,j=1}^3 \partial^{\alpha^\prime}(\partial_{x_i}u_j\partial_{x_j}u_i).
$$
We thus infer from \eqref{3.10} that for $0\leq\ell<{3\over 2}$ there holds,
\begin{equation}\label{3.14}
\begin{aligned}
 \norm{\comii{x}^\ell\nabla\partial^\alpha p}_{L^2} &\leq C\norm{\comii{x}^\ell \sum_{i,j=1}^3 \partial^{\alpha^\prime}(\partial_{x_i}u_j\partial_{x_j}u_i)}_{L^2}\\
   &\leq C\sum_{i,j=1}^3\norm{\comii{x}^\ell \partial^{\alpha^\prime}(\partial_{x_i}u_j\partial_{x_j}u_i)}_{L^2}\\
   &\leq C\norm{u}_{H^r}\norm{u}_{H^r_\ell},
\end{aligned}
\end{equation}
The last step in \eqref{3.14} with notice $\abs{\alpha^\prime}\leq r-1$ is very similar to the estimate of Lemma \ref{lem3.3} after expansion by Leibniz formula, we thus omit the details.
\end{proof}

With these Lemma \ref{lem3.3} and Lemma \ref{lem3.4} in preparation, we now prove the local existence and uniqueness of $H^r_\ell$- solutions to Euler equation.
\begin{proof}[Proof of (1) of Theorem \ref{Th2.1}]
We first obtain the a priori estimate. Suppose there exists smooth solutions $u, p$ (decaying polynomially at infinity) to the Euler equations, it is obvious that the zero order energy estimate holds,
\begin{equation*}
 \frac{1}{2}\frac{d}{dt}\norm{u(t)}_{L^2}^2=0.
\end{equation*}
For $\alpha\in \mathbb{N}^3$ with $1\leq\abs{\alpha}\leq r$, we apply $\partial^\alpha$ on the Euler equation to obtain
\begin{equation}\label{3.16}
 \partial_t \partial^\alpha u+\partial^\alpha(u\cdot\nabla u)+\nabla\partial^\alpha p=0.
\end{equation}
Taking $L^2$ inner product with $\comii{x}^{2\ell}\partial^\alpha u$ on both sides of \eqref{3.16}, we obtain,
\begin{equation*}
 \frac{d}{dt}\frac{1}{2}\norm{\comii{x}^\ell \partial^\alpha u}_{L^2}^2+\comii{\comii{x}^\ell\partial^\alpha(u\cdot\nabla u),\comii{x}^\ell\partial^\alpha u}+\comii{\comii{x}^\ell \nabla\partial^\alpha p,\comii{x}^\ell\partial^\alpha u}=0.
\end{equation*}
Summing over $1\leq\abs{\alpha}\leq r$ and adding the zero order energy term, and obtain,
\begin{equation*}
\begin{aligned}
 \frac{d}{dt}\frac{1}{2}\norm{u(t)}_{H^r_\ell}^2 &+\sum_{1\leq\abs{\alpha}\leq r}\comii{\comii{x}^\ell\partial^\alpha(u\cdot\nabla u),\comii{x}^\ell\partial^\alpha u}\\
 &+\sum_{1\leq\abs{\alpha}\leq r}\comii{\comii{x}^\ell \nabla\partial^\alpha p,\comii{x}^\ell\partial^\alpha u}=0.
\end{aligned}
\end{equation*}
We note that the divergence of $u$ is zero, so the cancellation equality holds,
\begin{equation*}
 \comii{u\cdot\nabla(\comii{x}^\ell \partial^\alpha u),\comii{x}^\ell \partial^\alpha u}=0.
\end{equation*}
Then the highest order term in the convecting term is,
\begin{equation*}
\begin{split}
 \comii{\comii{x}^\ell u\cdot\nabla\partial^\alpha u,\comii{x}^\ell \partial^\alpha u}
 &=\sum_{i,j=1}^3\int_{\mathbb{R}^3} \comii{x}^{2\ell} u_i(\partial_{x_i}\partial^\alpha u_j)\partial^\alpha u_j dx\\
 &=-\comii{\comii{x}^\ell u\cdot\nabla\partial^\alpha u,\comii{x}^\ell \partial^\alpha u}\\
 &\quad -\sum_{i=1}^3\int_{\mathbb{R}^3} (\partial_{x_i}\comii{x}^{2\ell}) u_i\abs{\partial^\alpha u}^2 dx.
\end{split}
\end{equation*}
Therefore, from the fact $\nabla\comii{x}^\ell\leq C\comii{x}^\ell$ we have,
\begin{equation}\label{3.21}
\begin{split}
 \abs{\comii{\comii{x}^\ell u\cdot\nabla\partial^\alpha u,\comii{x}^\ell \partial^\alpha u}}
 &\leq {1\over2}\sum_{i}^3\int_{\mathbb{R}^3}\abs{(\partial_{x_i}\comii{x}^{2\ell}) u_i}\abs{\partial^\alpha u}^2 dx\\
 &\leq C\norm{u}_{L^\infty}\norm{u}_{H^r_\ell}^2.
\end{split}
\end{equation}
With the use of the Lemma \ref{lem3.3} and Lemma \ref{lem3.4}, we obtain,
\begin{equation*}
\begin{aligned}
 \frac{d}{dt}\frac{1}{2}\norm{u(t)}_{H^r_\ell}^2 &\leq C\norm{u}_{H^r_\ell}\sum_{1\leq\abs{\alpha}\leq r}\sum_{0\neq\beta\leq\alpha}{\alpha\choose\beta}\norm{\comii{x}^\ell\partial^\beta u\cdot\nabla\partial^{\alpha-\beta}u}_{L^2}\\
  &\quad +C\norm{u}_{L^\infty}\norm{u}_{H^r_\ell}^2+C\norm{u}_{H^r_\ell}\sum_{1\leq\abs{\alpha}\leq r}\norm{\comii{x}^\ell \nabla\partial^\alpha p}_{L^2}\\
  &\leq C\norm{u}_{H^r}\norm{u}_{H^r_\ell}^2+C\norm{u}_{L^\infty}\norm{u}_{H^r_\ell}^2.
\end{aligned}
\end{equation*}
By use of Sobolev inequality, we have,
\begin{equation}\label{3.23}
 \frac{d}{dt}\norm{u(t)}_{H^r_\ell} \leq C\norm{u(t)}_{H^r}\norm{u(t)}_{H^r_\ell}.
\end{equation}
As we already know that $u(t,x)\in C([0,T^\ast), H^r)$, where $T^*$ is the maximal time of existence of $H^r-$ solution.
Then, on one side the Grownwall's inequality implies,
\begin{equation}\label{3.24}
 \norm{u(t)}_{H^r_\ell} \leq \norm{u_0}_{H^r_\ell} \exp{\bigg(\int_0^t C\norm{u(s)}_{H^r}ds\bigg)},\ \ 0<t<T^*.
\end{equation}
On the other hand, the definition of the weighted Sobolev space with $\ell\geq0$ obviously implies,
\begin{equation*}
 \norm{u}_{H^r}\leq\norm{u}_{H^r_\ell}.
\end{equation*}
Thus we obtain from \eqref{3.23},
\begin{equation*}
 \frac{d}{dt}\norm{u(t)}_{H^r_\ell} \leq C\norm{u(t)}_{H^r_\ell}^2.
\end{equation*}
Therefore the following inequality also holds,
\begin{equation*}
 \norm{u(t)}_{H^r_\ell} \leq \frac{\norm{u_0}_{H^r_\ell}}{1-C\norm{u_0}_{H^r_\ell}t},
\end{equation*}
with $0<t<\frac{1}{C\norm{u_0}_{H^r_\ell}}<T^\ast$.

Once we have obtained the a priori estimate, we need to construct the approximate solutions $u^{(n)}, n\in\mathbb{N}$. Before that we first rewrite the Euler system into the following modified Euler system,
\begin{equation*}
\left\{
\begin{aligned}
 \partial_t u+u\cdot\nabla u+\nabla \Pi(u,u) &=0,\\
 u|_{t=0} &=u_0,
\end{aligned}
\right.
\end{equation*}
where $\Pi(u,v)$ is defined in the following way
\begin{equation*}
 -\Delta \Pi(u,v)=\sum_{i,j=1}^3 \partial_{x_i}u_j\partial_{x_j}v_i.
\end{equation*}
Noting that we did not demand that the divergence of u is zero in this modified Euler system. It is showed in \cite{BCD} that the solution $u$ to this modified Euler system is also the solution $(u, \nabla p=\nabla\Pi(u,u))$ of the standard Euler equation with the same divergence free initial data $u_0$, and conversely the solution of standard Euler equation satisfies this modified system. Thus we now focus on the modified Euler system, and we shall use the following iteration scheme to construct approximate solutions to the modified Euler system,
\begin{equation}\label{3.30}
\left\{
\begin{aligned}
 \partial_t u^{(n+1)}+u^{(n)}\cdot\nabla u^{(n+1)}+ \nabla \Pi(u^{(n)},u^{(n)})&=0,\\
  u^{(n+1)}\big|_{t=0}&=u_0,\,n=0,1,2,\ldots,
\end{aligned}
\right.
\end{equation}
where $u^{(0)}=u_0$ and $\Pi(u^{(n)},u^{(n)})$ satisfies the following condition,
\begin{equation*}
 -\Delta \Pi(u^{(n)},u^{(n)}) =\sum_{i,j=1}^3\partial_{x_i} u_j^{(n)}\partial_{x_j} u_i^{(n)}.
\end{equation*}
For given $u^{(n)}(t,x)\in L^\infty_{loc}(\mathbb{R}^+,H^r_\ell)$ with $r\geq3$, then Lemma \ref{lem3.4} indicates that the pressure $\nabla \Pi(u^{(n)},u^{(n)})(t,x)\in L^\infty_{loc}(R^+,H^r_\ell)$. And the fact that $H^r\hookrightarrow C^1$ and the Cauchy Lipschitz Theorem shows that the transport equation at least exists a solution $u^{(n+1)}(t,x)\in L^\infty_{loc}(\mathbb{R}^+,H^r_\ell)$. Moreover, applying $\partial^\alpha$ on both sides of \eqref{3.30} with $1\leq\abs{\alpha}\leq r$ and taking $L^2$-inner product with $\comii{x}^{2\ell}\partial^\alpha u^{(n+1)}$, we can obtain from the Lemma \ref{lem3.3} and Lemma \ref{lem3.4},
\begin{equation*}
\begin{split}
 \frac{d}{dt}\norm{u^{(n+1)}(t,\cdot)}_{H^r_\ell}^2 &\leq C\norm{u^{(n)}(t,\cdot)}_{H^r_\ell}\norm{u^{(n+1)}(t,\cdot)}_{H^r_\ell}^2\\
 &\quad+C\norm{u^{(n)}(t,\cdot)}_{H^r_\ell}^2\norm{u^{(n+1)}(t,\cdot)}_{H^r_\ell}.
\end{split}
\end{equation*}
Using Grownwall inequality, we have for all $n\in\mathbb{N}$ the following inequality holds
\begin{equation}\label{3.33}
 \norm{u^{(n+1)}(t,\cdot)}_{H^r_\ell} \leq e^{CU^n(t)}\bigg(\norm{u_0}_{H^r_\ell}+C\int_0^t e^{-CU^n(s)}\norm{u^{(n)}(s,\cdot)}_{H^r_\ell}^2ds  \bigg),
\end{equation}
where
$$
U^n(t)=\int_0^t \norm{u^{(n)}(s,\cdot)}_{H^r_\ell}ds.
$$
We hope to find a uniform bound for $\{u^{(n)}\}, n\in\mathbb{Z}$. Let us argue by induction, we fix a $T>0$ such that $2C\norm{u_0}_{H^r_\ell}T<1$ and obviously have
\begin{equation*}
 \norm{u^{(0)}(t)}_{H^r_\ell} \leq \frac{\norm{u_0}_{H^r_\ell}}{1-2Ct\norm{u_0}_{H^r_\ell}},\ \forall t\in [0,T].
\end{equation*}
We claim that
\begin{equation}\label{3.35}
 \norm{u^{(n)}(t)}_{H^r_\ell} \leq \frac{\norm{u_0}_{H^r_\ell}}{1-2Ct\norm{u_0}_{H^r_\ell}},\ \forall t\in [0,T], \forall n\in\mathbb{N}.
\end{equation}
Obviously, the choice of $T$ ensures the claim holds for $n=0$, let us suppose the claim holds for $n=k$, i.e.
\begin{equation*}
\begin{split}
 \norm{u^{(k)}(t)}_{H^r_\ell} &\leq \frac{\norm{u_0}_{H^r_\ell}}{1-2Ct\norm{u_0}_{H^r_\ell}}\\
\end{split}
\end{equation*}
We proceed to show the claim holds for $n=k+1$. Plugging the bound into the iterative inequality \eqref{3.33}, we have
\begin{equation*}
\begin{split}
 \norm{u^{(k+1)}(t)}_{H^r_\ell} &\leq e^{C\int_0^t \frac{\norm{u_0}_{H^r_\ell}}{1-2Cs\norm{u_0}_{H^r_\ell}}ds}\norm{u_0}_{H^r_\ell}\\
 &\quad+C\int_0^t e^{C\int_s^t \frac{\norm{u_0}_{H^r_\ell}}{1-2C\tau\norm{u_0}_{H^r_\ell}}d\tau}\frac{\norm{u_0}_{H^r_\ell}^2}{(1-2Cs\norm{u_0}_{H^r_\ell})^2}ds\\
 &\leq \frac{\norm{u_0}_{H^r_\ell}}{1-2Ct\norm{u_0}_{H^r_\ell}}.
\end{split}
\end{equation*}
Therefore we proved the claim \eqref{3.35} by induction. And we have $\{u^{(n)}\}, n\in\mathbb{N}$ is uniformly bounded in $L^\infty(0,T;H^r_\ell)$. We then are going to show that $\{u^{(n)}\}, n\in\mathbb{N}$ is a Cauchy sequence in $C(0,T;H^{r-1}_\ell)$. For this purpose, we note that for all $(n,k)\in\mathbb{N}^2$, we have from the iteration scheme,
\begin{equation*}
\begin{aligned}
(\partial_t +u^{(n+k)}\cdot\nabla)u^{(n+k+1)}+\nabla \Pi(u^{(n+k)}, u^{(n+k)}) &=0,\\
(\partial_t +u^{(n)}\cdot\nabla)u^{(n+1)}+\nabla \Pi(u^{(n)}, u^{(n)}) &=0.
\end{aligned}
\end{equation*}
Taking difference of these two equations, we have,
\begin{equation}\label{3.39}
\begin{aligned}
 (\partial_t +u^{(n+k)}\cdot\nabla)(u^{(n+k+1)}-u^{(n+1)}) &=(u^{(n)}-u^{(n+k)})\cdot\nabla u^{(n+1)}\\
 &\quad+\nabla \Pi(u^{(n)}, u^{(n)})-\nabla \Pi(u^{(n+k)}, u^{(n+k)})
\end{aligned}
\end{equation}
We now want to estimate $\norm{u^{(n+k+1)}(t,\cdot)-u^{(n+1)}(t,\cdot)}_{H^{r-1}_\ell}$. For this purpose, we take the same procedure as in the estimate of the a priori estimate. From the Lemma \ref{lem3.3} we can obtain,
\begin{equation}\label{3.40}
\begin{split}
 &\comii{u^{(n+k)}\cdot \nabla (u^{(n+k+1)}-u^{(n+1)}), u^{(n+k+1)}-u^{(n+1)}}_{H^{r-1}_\ell}\\
 &\leq C\norm{u^{(n+k)}}_{H^r_\ell}\norm{u^{(n+k+1)}-u^{(n+1)}}_{H^{r-1}_\ell}^2,
\end{split}
\end{equation}
and
\begin{equation}\label{3.41}
\begin{split}
 &\comii{(u^{(n)}-u^{(n+k)})\cdot \nabla u^{(n+1)}, u^{(n+k+1)}-u^{(n+1)}}_{H^{r-1}_\ell}\\
 &\leq C\norm{u^{(n+1)}}_{H^r_\ell}\norm{u^{(n)}-u^{(n+k)}}_{H^{r-1}_\ell}\norm{u^{(n+k+1)}-u^{(n+1)}}_{H^{r-1}_\ell}.
\end{split}
\end{equation}
Noting that $\Pi(u^{(n)},u^{(n)})-\Pi(u^{(n+k)},u^{(n+k)})$ satisfies the following form,
\begin{equation}
\begin{split}
 &-\Delta \Pi(u^{(n)}, u^{(n)})+\Delta \Pi(u^{(n+k)},u^{(n+k)})\\
 &=\sum_{i,j=1}^3 \partial_{x_i}u_j^{(n)}\partial_{x_j}u_i^{(n)}-\sum_{i,j=1}^3 \partial_{x_i}u_j^{(n+k)}\partial_{x_j}u_i^{(n+k)}\\
 &=\sum_{i,j=1}^3 \partial_{x_i}(u_j^{(n)}-u_j^{(n+k)})\partial_{x_j}u_i^{(n)}+\sum_{i,j=1}^3 \partial_{x_i}u_j^{(n+k)}\partial_{x_j}(u_i^{(n)}-u_i^{(n+k)}).
\end{split}
\end{equation}
Hence by Lemma \ref{lem3.4}, we can obtain,
\begin{equation}\label{3.43}
\begin{split}
 &\comii{\nabla\big(\Pi(u^{(n)},u^{(n)})-\Pi(u^{(n+k)},u^{(n+k)})\big),u^{(n+k+1)}-u^{(n+1)}}_{H^{r-1}_\ell}\\
 &\leq C\bigg(\norm{u^{(n)}}_{H^r_\ell}+\norm{u^{(n+k)}}_{H^r_\ell}\bigg)\norm{u^{(n)}-u^{(n+k)}}_{H^{r-1}_\ell}
 \norm{u^{(n+1)}-u^{(n+k+1)}}_{H^{r-1}_\ell}
\end{split}
\end{equation}
Taking the $H^{r-1}_\ell$-inner product with $\big(u^{(n+k+1)}-u^{(n)}\big)$ on both sides of \eqref{3.39} and
combining \eqref{3.40}, \eqref{3.41} and \eqref{3.43}, we can have,
\begin{equation*}
\begin{aligned}
 &\frac{1}{2}\frac{d}{dt} \norm{u^{(n+k+1)}-u^{(n+1)}}_{H^{r-1}_\ell}^2 \\
 &\leq C\norm{u^{(n+k)}}_{H^r_\ell}\norm{u^{(n+k+1)}-u^{(n+1)}}_{H^{r-1}_\ell}^2\\
 &\quad+C\norm{u^{(n+1)}}_{H^r_\ell}\norm{u^{(n)}-u^{(n+k)}}_{H^{r-1}_\ell}\norm{u^{(n+k+1)}-u^{(n+1)}}_{H^{r-1}_\ell}\\
 &\quad+C\bigg(\norm{u^{(n)}}_{H^r_\ell}+\norm{u^{(n+k)}}_{H^r_\ell}\bigg)\norm{u^{(n)}-u^{(n+k)}}_{H^{r-1}_\ell}
 \norm{u^{(n+1)}-u^{(n+k+1)}}_{H^{r-1}_\ell}
\end{aligned}
\end{equation*}
Then the Grownwall inequality shows,
\begin{equation*}
\begin{aligned}
 &\norm{u^{(n+k+1)}(t)-u^{(n+1)}(t)}_{H^{r-1}_\ell}\leq C e^{C U^{(n+k)}(t)}\\
 &\quad\times\int_0^t e^{-CU^{(n+k)}(s)}\norm{u^{(n+k)}(s)-u^{(n)}(s)}_{H^{r-1}_\ell}\bigg(\norm{u^{(n)}(s)}_{H^r_\ell}+\norm{u^{(n+1)}(s)}_{H^r_\ell}\\
 &\quad\quad+\norm{u^{(n+k)}(s)}_{H^r_\ell} \bigg)ds.
\end{aligned}
\end{equation*}
Since $\{u^{(n)}\}, {n\in\mathbb{N}}$ is bounded in $L^\infty(0,T;H^r_\ell)$, we can find a constant $C_T$, independent of n and k, and such that for all t in $[0,T]$, we have
$$
 \norm{u^{(n+k+1)}(t)-u^{(n+1)}(t)}_{H^{r-1}_\ell}\leq C_T\int_0^t \norm{u^{(n+k)}(s)-u^{(n)}(s)}_{H^{r-1}_\ell}ds.
$$
Hence, arguing by induction, we get
$$
 \sup_{t\in[0,T]}\norm{u^{(n+k+1)}(t)-u^{(n+1)}(t)}_{H^{r-1}_\ell}\leq \frac{(T C_T)^{n+1}}{(n+1)!}\sup_{t\in[0,T]}\norm{u^{(k)}(t)-u^{(0)}(t)}_{H^{r-1}_\ell}.
$$
Since $\sup_{t\in[0,T]}\norm{u^{(k)}(t)}_{H^r_\ell}$ is bounded independent of k, we can guarantee the existence of some new constant $C_T^\prime$ such that
$$
 \sup_{t\in[0,T]}\norm{u^{(n+k)}(t)-u^{(n)}(t)}_{H^{r-1}_\ell}\leq C_T^\prime \frac{(T C_T)^{n+1}}{(n+1)!} .
$$
Hence, $\{u^{(n)}\}_{n\in\mathbb{N}}$ is a Cauchy sequence in $C(0,T;H^{r-1}_\ell)$ and converges to some limit function $u\in C(0,T; H^{r-1}_\ell)$. We have to check that $u$ belongs to $L^\infty(0,T;H^r_\ell)$ and satisfies Euler equation. Since $\{u^{(n)}\}_{n\in\mathbb{N}}$ is bounded in $L^\infty(0,T;H^r_\ell)$, the Fatou property for Sobolev space guarantees that u also belongs to $L^\infty(0,T;H^r_\ell)$. Now, as $(u^n)_{n\in\mathbb{N}}$ converges to u in $C(0,T;H^{r-1}_\ell)$, it is then easy to pass to the limit in \eqref{3.30} and to concludes that u is indeed a solution of the modified Euler equation. Since $u_0$ is divergence free, we have $u$ is divergence free. Let $\nabla p=\nabla \Pi(u,u)$ which is defined by \eqref{3.31}, we finally obtain that $u, \nabla p$ is the solution of \eqref{1.1}. Then the first part $(1)$ of Theorem \ref{Th2.1} is proved.
\end{proof}

\section{Weighted Gevrey-class regularity of Euler equation}\label{Sec4}
In this section, we will consider the weighted Gevrey-class regularity of Euler equation. It is showed in \cite{KV} that the solution remains in Gevrey-class if the initial data was so, and the decay of the radius of Gevrey-class regularity can also be obtained explicitly. In the following, we will show that the weighted Gevrey-class regularity also propagate and the radius of weighted Gevrey-class can be obtained explicitly, i.e. we will prove the second part (2) of the main Theorem \ref{Th2.1}.


\begin{proof}[Proof of (2) Theorem \ref{Th2.1}]
Since the initial data $u_0$ is of weighted Gevrey class s, then there exists $\tau(0)>0$ such that $u_0\in X_{\tau_0,\ell}$. Let $u(t,x)$ be of weighted Gevrey class s and also the $H^r_\ell-$ solution obtained in the previous section and suppose $\tau(t)$ is a smooth function of $t$, then we have
\begin{equation}\label{4.1}
 \frac{d}{dt}\norm{u(t)}_{X_{\tau(t),\ell}}=\dot{\tau}\norm{u(t)}_{Y_{\tau,\ell}}+
 \sum_{m=3}^\infty\frac{d}{dt}\abs{u(t)}_{m,\ell}\frac{\tau^{m-3}}{(m-3)!^s}.
\end{equation}
Going back to the Equation,
$$
 \partial_t u+u\cdot\nabla u+\nabla p=0.
$$
Applying $\partial^\alpha$ with $\abs{\alpha}=m$, taking the $L^2-$ inner product with $\comii{x}^{2\ell}\partial^\alpha u(x)$, we obtain
$$
 \comii{\partial_t \partial^\alpha u,\comii{x}^{2\ell}\partial^\alpha u}+\comii{\partial^\alpha(u\cdot\nabla u),\comii{x}^{2\ell}\partial^\alpha u}+\comii{\nabla\partial^\alpha p,\comii{x}^{2\ell}\partial^\alpha u}=0.
$$
From \eqref{3.21}, we have
$$
 \abs{\comii{\comii{x}^\ell u\cdot\nabla\partial^\alpha u,\comii{x}^\ell\partial^\alpha u}}\leq C\norm{u}_{L^\infty}\norm{\comii{x}^\ell \partial^\alpha u}_{L^2}^2,
$$
where we have used the fact that $\nabla\comii{x}^\ell\leq C\comii{x}^\ell$.
We then have
\begin{equation*}
\begin{aligned}
 \frac{d}{dt}\norm{\comii{x}^\ell \partial^\alpha u(t)}_{L^2} &\leq \sum_{0\neq\beta\leq\alpha}{\alpha\choose\beta}\norm{\comii{x}^\ell \partial^\beta u\cdot\nabla\partial^{\alpha-\beta}u}_{L^2}+\norm{\comii{x}^\ell\nabla\partial^\alpha p}_{L^2}\\
 &\quad+C\norm{u}_{L^\infty}\norm{\comii{x}^\ell \partial^\alpha u}_{L^2}.
\end{aligned}
\end{equation*}
Summing over $\abs{\alpha}=m$ yields
\begin{equation*}
\begin{aligned}
 \frac{d}{dt}\abs{u(t)}_{m,\ell} &\leq \sum_{\abs{\alpha}=m}\sum_{0\neq\beta\leq\alpha}{\alpha\choose\beta}\norm{\comii{x}^\ell \partial^\beta u\cdot\nabla\partial^{\alpha-\beta}u}_{L^2}+\sum_{\abs{\alpha}=m}\norm{\comii{x}^\ell\nabla\partial^\alpha p}_{L^2}\\
  &\quad+C\norm{u}_{L^\infty}\sum_{\abs{\alpha}=m}\norm{\comii{x}^\ell \partial^\alpha u}_{L^2}.
\end{aligned}
\end{equation*}
Plugging into \eqref{4.1} to obtain
\begin{equation}\label{4.4}
\begin{aligned}
 \frac{d}{dt}\norm{ u(t)}_{X_{\tau(t),\ell}} &\leq \dot{\tau}\norm{ u(t)}_{Y_{\tau,\ell}}+\sum_{m=3}^\infty\frac{d}{dt}\abs{u(t)}_{m,\ell}\frac{\tau^{m-3}}{(m-3)!^s}\\
  &\leq \dot{\tau}\norm{ u(t)}_{Y_{\tau,\ell}}+\mathcal{C}_\ell+\mathcal{P}_\ell
  +C\norm{u}_{L^\infty}\sum_{m=3}^\infty\abs{u(t)}_{m,\ell}\frac{\tau^{m-3}}{(m-3)!^s}\\
  &\leq \dot{\tau}\norm{ u(t)}_{Y_{\tau,\ell}}+\mathcal{C}_\ell+\mathcal{P}_\ell+C\norm{u}_{L^\infty}\abs{ u}_{3,\ell}+C\tau\norm{u}_{L^\infty}
  \norm{ u}_{Y_{\tau,\ell}}
\end{aligned}
\end{equation}
where
$$
 \mathcal{C}_\ell=\sum_{m=3}^\infty\sum_{\abs{\alpha}=m}\sum_{0\neq\beta\leq\alpha}{\alpha\choose\beta}
 \norm{\comii{x}^\ell\partial^\beta u\cdot\nabla \partial^{\alpha-\beta}u}_{L^2}\frac{\tau^{m-3}}{(m-3)!^s}
$$
and
$$
 \mathcal{P}_\ell=\sum_{m=3}^\infty\sum_{\abs{\alpha}=m}\norm{\comii{x}^\ell\nabla\partial^\alpha p}_{L^2}\frac{\tau^{m-3}}{(m-3)!^s}
$$
It remains to estimate $\mathcal{C}_\ell$ and $\mathcal{P}_\ell$. We will follow the argument of Kukavica and Vicol in their work \cite{KV} to estimate $\mathcal{C}_\ell$ and $\mathcal{P}_\ell$. And the estimate of $\mathcal{P}_\ell$ follows from the consequence of Calder\'on-Zygmund theory with $\mathcal{A}_2$ weights(see \cite{SEM}). Let us first state a Lemma which will be used throughout the estimate of $\mathcal{C}_\ell$ and $\mathcal{P}_\ell$.
\begin{lemma}
Let $\{x_\lambda\}_{\lambda\in\mathbb{N}^3}$ and $\{y_\lambda\}_{\lambda\in\mathbb{N}^3}$ be real numbers, then the following identity holds,
\begin{equation}\label{4.5}
\sum_{\abs{\alpha}=m}\sum_{\abs{\beta}=j,\beta\leq\alpha}x_{\beta}y_{\alpha-\beta} =\bigg(\sum_{\abs{\beta}=j}x_\beta\bigg) \bigg(\sum_{\abs{\gamma}=m-j}y_\gamma\bigg).
\end{equation}
\end{lemma}
The proof is trivial by relabeling the multi-indexes, we thus omit the details. Another fact shall be used in the following is that,
\begin{equation*}
 {\alpha\choose\beta}\leq {\abs{\alpha}\choose\abs{\beta}},
\end{equation*}
holds for $\beta\leq\alpha\in\mathbb{N}_0^3$.
\begin{lemma}\label{lem4.2}
The estimate of $\mathcal{C}_\ell$ satisfies the following form
\begin{equation*}
\begin{aligned}
 \mathcal{C}_\ell &\leq C\norm{u}_{H^r_\ell}^2(1+\tau^2)+C\tau\norm{u}_{H^r}\norm{u}_{Y_{\tau,\ell}}+
 C\tau^{3/2}\norm{u}_{X_\tau}\norm{u}_{Y_{\tau,\ell}}\\
 &+C\tau^2\norm{u}_{H^r}\norm{u}_{Y_{\tau,\ell}}+C\tau^3\norm{u}_{H^r}\norm{u}_{Y_{\tau,\ell}}.
\end{aligned}
\end{equation*}
\end{lemma}
\begin{proof}
Inspired by \eqref{4.5}, if we denote
$$
 \mathcal{C}_{\ell,m,j}=\frac{\tau^{m-3}}{(m-3)!^s}\sum_{\abs{\alpha}=m}\sum_{\abs{\beta}=j,\beta\leq\alpha}{\alpha\choose\beta}
 \norm{\comii{x}^\ell\partial^\beta u\cdot\nabla\partial^{\alpha-\beta}u}_{L^2},
$$
then the summation of $\mathcal{C}_\ell$ can be rewritten in the following form
$$
 \mathcal{C}_\ell=\sum_{m=3}^\infty \sum_{j=1}^m \mathcal{C}_{\ell,m,j}.
$$
We then divide the right side of the above equality into seven terms according to the values of m and j,
\begin{equation}\label{4.8}
\begin{split}
 \mathcal{C}_\ell &=\sum_{m=3}^\infty \mathcal{C}_{\ell,m,1}+\sum_{m=3}^\infty \mathcal{C}_{\ell,m,2}+\sum_{m=6}^\infty \sum_{j=3}^{[m/2]}\mathcal{C}_{\ell,m,j}+\sum_{m=7}^\infty \sum_{j=[m/2]+1}^{m-3}\mathcal{C}_{\ell,m,j}\\
 &\quad+\sum_{m=5}^\infty \mathcal{C}_{\ell,m,m-2}+\sum_{m=4}^\infty \mathcal{C}_{\ell,m,m-1}+\sum_{m=3}^\infty \mathcal{C}_{\ell,m,m}.
\end{split}
\end{equation}
The we are supposed to estimate the right hand of \eqref{4.8} in terms of the Sobolev norms and Gevrey norms. Consequently, we have for j=1,
\begin{equation}\label{4.9}
\begin{split}
 \sum_{m=3}^\infty \mathcal{C}_{\ell,m,1} &\leq \sum_{m=3}^\infty \frac{m\tau^{m-3}}{(m-3)!^s} \abs{u}_{1,\infty} \abs{u}_{m,\ell}\\
 &\leq C\abs{u}_{1,\infty}\abs{u}_{3,\ell}+C\tau\abs{u}_{1,\infty}\norm{u}_{Y_{\tau,\ell}}\\
 &\leq C\norm{u}_{H^r_\ell}^2+C\tau \norm{u}_{H^r} \norm{u}_{Y_{\tau,\ell}},
\end{split}
\end{equation}
where we used the Sobolev embedding inequality in the last estimate of the above inequality. When $j=2$, we have,
\begin{equation}\label{4.10}
\begin{split}
 \sum_{m=3}^\infty \mathcal{C}_{\ell,m,2} &\leq \sum_{m=3}^\infty \frac{\tau^{m-3}}{(m-3)!^s}{m\choose 2}\abs{u}_{2,\infty} \abs{u}_{m-1,\ell}\\
 &\leq 3\abs{u}_{2,\infty}\abs{u}_{2,\ell}+6\tau \abs{u}_{2,\infty}\abs{u}_{3,\ell}\\
 &\quad+\tau^2\abs{u}_{2,\infty}\sum_{m=5}^\infty \abs{u}_{m-1,\ell}\frac{(m-4)\tau^{m-5}}{(m-4)!^s}\bigg({m\choose2}\frac{1}{(m-4)(m-3)^s}\bigg)\\
 &\leq C\norm{u}_{H^r_\ell}^2+C\tau\norm{u}_{H^r_\ell}^2+C\tau^2\norm{u}_{H^r}\norm{u}_{Y_{\tau,\ell}},
\end{split}
\end{equation}
where we use the fact that there exists a constant $C$ such that
$$\bigg({m\choose2}\frac{1}{(m-4)(m-3)^s}\bigg)\leq C,$$
for all $m\geq5$. When $j$ varies from $3$ to $m-3$, we shall use the following Sobolev inequality,
$$
 \norm{u}_{L^\infty} \leq C\norm{u}_{L^2}^{1/4}\norm{D^2 u}_{L^2}^{3/4}.
$$
For example when $j$ varies from $3$ to $[m/2]$, we have,
\begin{equation}\label{4.11}
\begin{split}
 \sum_{m=6}^\infty \sum_{j=3}^{[m/2]}\mathcal{C}_{\ell,m,j} &\leq \sum_{m=6}^\infty \sum_{j=3}^{[m/2]} \frac{\tau^{m-3}}{(m-3)!^s}{m\choose j}\abs{u}_{j,\infty}\abs{u}_{m-j+1,\ell}\\
 &\leq C\sum_{m=6}^\infty \sum_{j=3}^{[m/2]} \frac{\tau^{m-3}}{(m-3)!^s}{m\choose j}\abs{u}_{j}^{1/4}\abs{u}_{j+2}^{3/4}\abs{u}_{m-j+1,\ell}\\
 &\leq C\tau^{3/2}\sum_{m=6}^\infty\sum_{j=3}^{[m/2]}\bigg[ \bigg(\abs{u}_j\frac{\tau^{j-3}}{(j-3)!^s}\bigg)^{1/4}\bigg(\abs{u}_{j+2}\frac{\tau^{j-1}}{(j-1)!^s}\bigg)^{3/4}\\
 &\quad\times\bigg(\abs{u}_{m-j+1,\ell}\frac{(m-j-2)\tau^{m-j-3}}{(m-j-2)!^s}\bigg)\mathcal{A}_{m,j,s}\bigg]\\
 &\leq C\tau^{3/2}\norm{u}_{X_\tau}\norm{u}_{Y_{\tau,\ell}},
\end{split}
\end{equation}
where
$$\mathcal{A}_{m,j,s}={m\choose j}\frac{(m-j-2)!^s (j-1)!^{3s/4}(j-3)!^{s/4}}{(m-j-2)(m-3)!^s}$$
is bounded by some constant $C$ for $3\leq j\leq m/2, s\geq1$. One can justify this fact by expressing $\mathcal{A}_{m,j,s}$ as
\begin{equation*}
 \mathcal{A}_{m,j,s} ={{m-3}\choose{j-1}}^{1-s} \frac{m(m-1)(m-2)}{j(m-j)(m-j-1)(m-j-2)(j-1)^{s/4}(j-2)^{s/4}}.
\end{equation*}
One then easily see that for $3\leq j\leq m/2, s\geq1$,
\begin{equation*}
 \mathcal{A}_{m,j,s} \lesssim \frac{1}{j(j-1)^{s/4}(j-2)^{s/4}}\leq C.
\end{equation*}
When $j$ varies from $[m/2]+1$ to $m-3$, we can symmetrically have,
\begin{equation}\label{4.12}
\begin{split}
 \sum_{m=7}^\infty \sum_{j=[m/2]+1}^{m-3}\mathcal{C}_{\ell,m,j} &\leq \sum_{m=6}^\infty \sum_{j=[m/2]+1}^{m-3} \frac{\tau^{m-3}}{(m-3)!^s}{m\choose j}\abs{u}_{j,\ell}\abs{u}_{m-j+1,\infty}\\
 &\leq C\sum_{m=6}^\infty \sum_{j=[m/2]+1}^{m-3} \frac{\tau^{m-3}}{(m-3)!^s}{m\choose j}\abs{u}_{j,\ell}\abs{u}_{m-j+1}^{1/4}\abs{u}_{m-j+3}^{3/4}\\
 &\leq C\tau^{3/2}\sum_{m=6}^\infty\sum_{j=[m/2]+1}^{m-3} \mathcal{A}^\prime_{m,j,s}\bigg[ \bigg(\abs{u}_{j,\ell}\frac{(j-3)\tau^{j-4}}{(j-3)!^s}\bigg)\\
 &\quad\times\bigg(\abs{u}_{m-j+1}\frac{\tau^{m-j-2}}{(m-j-2)!^s}\bigg)^{1/4}
 \bigg(\abs{u}_{m-j+3,\ell}\frac{\tau^{m-j}}{(m-j)!^s}\bigg)^{3/4}\bigg]\\
 &\leq C\tau^{3/2}\norm{u}_{X_\tau}\norm{u}_{Y_{\tau,\ell}},
\end{split}
\end{equation}
where
$$\mathcal{A}^\prime_{m,j,s}={m\choose j}\frac{(j-3)!^s (m-j-2)!^{s/4}(m-j)!^{3s/4}}{(j-3)(m-3)!^s}$$
is also bounded by some constant $C$ for $m/2\leq j\leq m-3, s\geq1$ because it can be expressed as
\begin{equation*}
 \mathcal{A}^\prime_{m,j,s} ={{m-3}\choose{j-3}}^{1-s} \frac{m(m-1)(m-2)}{j(j-1)(j-2)(j-3)(m-j-1)^{s/4}(m-j)^{s/4}}<C.
\end{equation*}
When $j=m-2$, we have
\begin{equation}\label{4.13}
\begin{split}
 \sum_{m=5}^\infty \mathcal{C}_{\ell,m,m-2} &\leq C\abs{u}_{3,\infty}\abs{u}_{3,\ell}\tau^2+C\abs{u}_{3,\infty}\tau^3\sum_{m=6}^\infty \abs{u}_{m-2,\ell} \frac{(m-5)\tau^{m-6}}{(m-5)!^s}\\
 &\quad\times{m\choose 2}\frac{1}{(m-5)(m-3)^s(m-4)^s}\\
 &\leq C\norm{u}_{H^r_\ell}^2\tau^2+C\tau^3\norm{u}_{H^r}\norm{u}_{Y_{\tau,\ell}}.
\end{split}
\end{equation}
When $j=m-1$, we similarly have,
\begin{equation}\label{4.14}
\begin{split}
 \sum_{m=4}^\infty \mathcal{C}_{m,m-1} &\leq \sum_{m=4}^\infty m \abs{u}_{2,\infty}\abs{u}_{m-1,\ell}\frac{\tau^{m-3}}{(m-3)!^s}\\
 &\leq C\tau\abs{u}_{2,\infty}\abs{u}_{3,\ell}+C\tau^2\abs{u}_{2,\infty}\norm{u}_{Y_{\tau,\ell}}\\
 &\leq C\norm{u}_{H^r_\ell}^2\tau+C\tau^2\norm{u}_{H^r}\norm{u}_{Y_{\tau,\ell}}.
\end{split}
\end{equation}
Lastly, we have
\begin{equation}\label{4.15}
\begin{split}
 \sum_{m=3}^\infty \mathcal{C}_{m,m} &\leq \sum_{m=3}^\infty  \abs{u}_{1,\infty}\abs{u}_{m,\ell}\frac{\tau^{m-3}}{(m-3)!^s}\\
 &\leq C\abs{u}_{1,\infty}\abs{u}_{3,\ell}+C\tau\abs{u}_{1,\infty}\norm{u}_{Y_{\tau,\ell}}\\
 &\leq C\norm{u}_{H^r_\ell}^2+C\tau\norm{u}_{H^r}\norm{u}_{Y_{\tau,\ell}}.
\end{split}
\end{equation}
Plugging the estimates \eqref{4.9}-\eqref{4.15} into \eqref{4.8}, we then prove the Lemma \ref{lem4.2}.
\end{proof}

\begin{lemma}\label{lem4.3}
The estimate of $\mathcal{P}_\ell$ satisfies the following form
\begin{equation*}
\begin{aligned}
 \mathcal{P}_\ell &\leq C\norm{u}_{H^r_\ell}^2+C\tau\norm{u}_{H^r}\big(\norm{u}_{H^r_\ell}+\norm{u}_{Y_{\tau,\ell}}\big)+
 C\tau^{3/2}\norm{u}_{X_\tau}\norm{u}_{Y_{\tau,\ell}}\\
 &+C\tau^2\norm{u}_{H^r}\big(\norm{u}_{H^r_\ell}+\norm{u}_{Y_{\tau,\ell}}\big)
 +C\tau^3\norm{u}_{H^r}\norm{u}_{Y_{\tau,\ell}}.
\end{aligned}
\end{equation*}
\end{lemma}
\begin{proof}
Using inequality \eqref{3.10}, the summation can first be bounded by
\begin{equation}\label{4.20+}
\begin{aligned}
 \mathcal{P}_\ell &=\sum_{m=3}^\infty\sum_{\abs{\alpha}=m}\norm{\comii{x}^\ell\partial^\alpha \nabla p}_{L^2}\frac{\tau^{m-3}}{(m-3)!^s}\\
 &=\sum_{m=3}^\infty\sum_{\abs{\alpha}=m}\norm{\comii{x}^\ell\partial^{\alpha-\alpha^\prime}\nabla \partial^{\alpha^\prime} p}_{L^2}\frac{\tau^{m-3}}{(m-3)!^s},\text{for some}\ \alpha^\prime\leq\alpha\ \text{with}\ \abs{\alpha-\alpha^\prime}=1 \\
 &\leq 3C\sum_{m=3}^\infty \sum_{\abs{\gamma}=m-1}\sum_{i,j=1}^3\norm{\comii{x}^\ell \partial^{\gamma} (\partial_{x_j}u_i\partial_{x_i}u_j)}_{L^2}\frac{\tau^{m-3}}{(m-3)!^s},
\end{aligned}
\end{equation}
where we have used
$$
 -\Delta \partial^{\alpha^\prime}p=\sum_{i,j=1}^3 \partial^{\alpha^\prime}\big(\partial_{x_j}u_i\partial_{x_i}u_j\big).
$$
Thus
$$
 \norm{\comii{x}^\ell\partial^{\alpha-\alpha^\prime}\nabla \partial^{\alpha^\prime} p}_{L^2}\leq C\sum_{i,j=1}^3\norm{\comii{x}^\ell \partial^{\alpha^\prime} (\partial_{x_j}u_i\partial_{x_i}u_j)}_{L^2}
$$
We now want to estimate the right hand side of \eqref{4.20+}, at first we rewrite the right hand side summation in the following way (still denote by $\alpha$ in the summation),
$$
 \mathcal{P}_\ell \leq C\sum_{m=3}^\infty \sum_{\abs{\alpha}=m-1}\sum_{\beta\leq\alpha}\sum_{i,j=1}^3 {\alpha\choose\beta}\norm{\comii{x}^\ell (\partial^\beta\partial_{x_j}u_i)(\partial^{\alpha-\beta}\partial_{x_i}u_j)}_{L^2}\frac{\tau^{m-3}}{(m-3)!^s}.
$$
If denote
$$
 \mathcal{P}_{\ell,m,k}=\frac{\tau^{m-3}}{(m-3)!^s}\sum_{\abs{\alpha}=m-1}\sum_{\abs{\beta}=k,\beta\leq\alpha}{\alpha\choose\beta}
 \sum_{i,j=1}^3\norm{\comii{x}^\ell (\partial^\beta \partial_{x_j}u_i)(\partial^{\alpha-\beta}\partial_{x_i}u_j)}_{L^2}.
$$
Then the right hand side can be written as
\begin{equation}\label{4.20}
\begin{split}
 \mathcal{P}_\ell &\leq C\sum_{m=3}^\infty \sum_{k=0}^{m-1} \mathcal{P}_{\ell,m,k}\\
  &\leq C\sum_{m=3}^\infty \mathcal{P}_{\ell,m,0}+C\sum_{m=3}^\infty \mathcal{P}_{\ell,m,1}+C\sum_{m=5}^\infty \mathcal{P}_{\ell,m,2}+C\sum_{m=8}^\infty\sum_{j=3}^{[m/2]-1}\mathcal{P}_{\ell,m,k}\\
  &\quad+C\sum_{m=6}^\infty\sum_{j=[m/2]}^{m-3}\mathcal{P}_{\ell,m,k}+C\sum_{m=4}^\infty \mathcal{P}_{\ell,m,m-2}+C\sum_{m=3}^\infty \mathcal{P}_{\ell,m,m-1},
\end{split}
\end{equation}
It rests to estimate the right hand side of \eqref{4.20}. Since they are quite similar with the previous Lemma \ref{lem4.3}, we list the results here for simplification.
\begin{equation}\label{4.21}
\begin{aligned}
 \sum_{m=3}^\infty \mathcal{P}_{\ell,m,0} &\leq C\norm{u}_{H^r_\ell}^2+C\tau\norm{u}_{H^r}\norm{u}_{Y_{\tau,\ell}},\\
 \sum_{m=3}^\infty \mathcal{P}_{\ell,m,1} &\leq C(1+\tau)\norm{u}_{H^r_\ell}^2+C\tau^2\norm{u}_{H^r}\norm{u}_{Y_{\tau,\ell}},\\
 \sum_{m=5}^\infty \mathcal{P}_{\ell,m,2} &\leq C\tau^2\norm{u}_{H^r_\ell}^2+C\tau^3\norm{u}_{H^r}\norm{u}_{Y_{\tau,\ell}},\\
 \sum_{m=8}^\infty \sum_{k=3}^{[m/2]-1}\mathcal{P}_{\ell,m,k} &\leq C\tau^{3/2}\norm{u}_{X_\tau}\norm{u}_{Y_{\tau,\ell}},\\
 \sum_{m=6}^\infty \sum_{j=[m/2]}^{m-3}\mathcal{P}_{\ell,m,j} &\leq C\tau^{3/2}\norm{u}_{X_\tau}\norm{u}_{Y_{\tau,\ell}},\\
 \sum_{m=4}^\infty \mathcal{P}_{\ell,m,m-2} &\leq C\tau\norm{u}_{H^r_\ell}^2+C\tau^2\norm{u}_{H^r}\norm{u}_{Y_{\tau,\ell}},\\
 \sum_{m=3}^\infty \mathcal{P}_{\ell,m,m-1} &\leq C\norm{u}_{H^r_\ell}^2+C\tau\norm{u}_{H^r}\norm{u}_{Y_{\tau,\ell}}.
\end{aligned}
\end{equation}
Substituting the right hand side estimates of \eqref{4.21}, we then conclude the proof of Lemma \ref{lem4.3}.
\end{proof}
For $r\geq5$ fixed, we use the Sobolev embedding theorem, and Lemma \ref{lem4.2} and Lemma \ref{lem4.3} to infer from \eqref{4.4},
\begin{equation}\label{4.23}
\begin{aligned}
 &\frac{d}{dt}\norm{u(t)}_{X_{\tau(t),\ell}} \leq C\norm{u(t)}_{H^r_\ell}^2\big(1+\tau(t)^2\big)+\dot{\tau}(t)\norm{u(t)}_{Y_{\tau(t),\ell}}\\
 &+C\norm{u(t)}_{Y_{\tau,\ell}} \bigg(\tau(t)\norm{u(t)}_{H^r}+(\tau(t)^2+\tau(t)^3)\norm{u(t)}_{H^r}+\tau(t)^{3/2}\norm{u(t)}_{X_{\tau(t)}}\bigg).
\end{aligned}
\end{equation}
If $\tau(t)$ decreases fast enough so that for all $0\leq t<T^*$ we have,
\begin{equation}\label{4.24}
 \dot\tau(t)+C\tau(t)\norm{u(t)}_{H^r}+C\big(\tau(t)^2+\tau(t)^3\big)\norm{u(t)}_{H^r}+C\tau(t)^{3/2}\norm{u(t)}_{X_{\tau(t)}}
 \leq0.
\end{equation}
Then \eqref{4.23} and the fact $\tau(t)\leq \tau(0)$ imply that
$$
 \frac{d}{dt}\norm{u(t)}_{X_{\tau(t),\ell}}\leq C\norm{u(t)}_{H^r_\ell}^2(1+\tau(0)^2).
$$
Integrating from $0$ to $t$, we have from \eqref{3.24}
\begin{equation}\label{4.25}
 \norm{u(t)}_{X_{\tau(t),\ell}} \leq \norm{u_0}_{X_{\tau(0),\ell}}+C_{\tau(0)}\int_0^t \norm{u(s)}_{H^r_\ell}^2 ds
\end{equation}
for all $0\leq t<T^*$, where $C_{\tau(0)}=1+\tau(0)^2$. We denote by $H(t)$ the right hand side of \eqref{4.25}, and it follows from the inequality \eqref{3.24},
\begin{equation*}
\begin{aligned}
 H(t) &:=\norm{u_0}_{X_{\tau(0),\ell}} +C_{\tau(0)} \int_0^t \norm{u(s)}_{H^r_\ell}^2 ds \\
 &\leq \norm{u_0}_{X_{\tau(0),\ell}} +C_{\tau(0)} t\norm{u_0}_{H^r_\ell}^2  \exp{\bigg(C\int_0^t \norm{u(s)}_{H^r}ds\bigg)}.
\end{aligned}
\end{equation*}
Since $\tau$ must be chosen to be a decreasing function, a sufficient condition for \eqref{4.24} to hold is that
\begin{equation}\label{4.27}
 \dot\tau(t)+2C\tau(t)\norm{u(t)}_{H^r}+2C\tau(t)^{3/2}\big(C_{\tau(0)}^\prime\norm{u(t)}_{H^r_\ell}+H(t)\big)=0
\end{equation}
where $C_{\tau(0)}^\prime=\tau(0)^{1/2}+\tau(0)^{3/2}$.
It then follows that if we solve the ODE \eqref{4.27} for $\tau(t)$,
\begin{equation}\label{4.28}
\begin{aligned}
 &\frac{1}{\tau(t)^{1/2}} =\exp{\bigg(C\int_0^t \norm{u(s)}_{H^r}ds\bigg)} \\ &\times\bigg[\tau(0)^{-1/2}+C\int_0^t\big(C_{\tau(0)}^\prime\norm{u(s)}_{H^r_\ell}+H(s)\big)\exp{\bigg(-C\int_0^s \norm{u(\lambda)}_{H^r}d\lambda\bigg)}ds\bigg]
\end{aligned}
\end{equation}
We note from \eqref{3.24} if the constant $C$ is large enough such that
$$
 \norm{u(t)}_{H^r_\ell}^2 \leq \norm{u_0}_{H^r_\ell}^2 \exp{\bigg(C\int_0^t \norm{u(s)}_{H^r}ds\bigg)}
$$
Then we have
\begin{equation*}
\begin{aligned}
 \tau(0)^{-1/2}+ &C\int_0^t\bigg(C_{\tau(0)}^\prime\norm{u(s)}_{H^r_\ell}+M(s)\bigg)G(s)^{-1}ds\\
   &\leq \tau(0)^{-1/2}+C\int_0^t\bigg(C_{\tau(0)}^\prime\norm{u_0}_{H^r_\ell}+\norm{u_0}_{X_{\tau(0),\ell}}+
   sC_{\tau(0)}\norm{u_0}_{H^r_\ell}^2\bigg)ds\\
   &\leq C_0(1+t)^2,
\end{aligned}
\end{equation*}
and therefore \eqref{4.28} implies
\begin{equation*}
 \frac{1}{\tau(t)^{1/2}}\leq C_0(1+t)^2\exp{\bigg(C\int_0^t \norm{u(s)}_{H^r}ds\bigg)}
\end{equation*}
We recall (see \cite{MB}) the $H^r$-norm of $u$ has an upper bound like
\begin{equation*}
 \norm{u(t)}_{H^r}\leq \frac{\norm{u_0}_{H^r}}{1-C_r\norm{u_0}_{H^r}t},\quad 0<t<T^\ast,
\end{equation*}
where $C_r$ depends on $r$ and we can enlarge it to be $C$. Then
\begin{equation*}
 \frac{1}{\tau(t)^{1/2}}\leq \frac{C_0(1+t)^2}{1-C\norm{u_0}_{H^r}t}
\end{equation*}
And thus we obtain the lower bound for $\tau$,
\begin{equation*}
 \tau(t)\geq \frac{(1-C\norm{u_0}_{H^r}t)^2}{C_0(1+t)^4}
\end{equation*}
In such case choice of $\tau$, we also have from \eqref{4.23},
\begin{equation}\label{4.30}
\begin{aligned}
 &\frac{d}{dt}\norm{u(t)}_{X_{\tau(t),\ell}}
 +C\norm{u(t)}_{Y_{\tau,\ell}} \bigg[\tau(t)\norm{u(t)}_{H^r}+C\tau(t)^{3/2}\big(C_{\tau(0)}^\prime\norm{u(t)}_{H^r_\ell}+H(t)\big)\bigg] \\
 &\leq C\norm{u(t)}_{H^r_\ell}^2\big(1+\tau(t)^2\big).\\
\end{aligned}
\end{equation}
Since $\tau(t)$ has a lower bound for sufficient small $0<T<T^\ast$, we then obtain by integrating \eqref{4.30} from $0$ to $T$,
$$
 \int_0^T \norm{u(s)}_{Y_{\tau(s),\ell}}<\infty.
$$
Thus we have $u(t,x)\in L^\infty([0,T),X_{\tau(\cdot),\ell})\cap L^1([0,T),Y_{\tau(\cdot),\ell})$.
This concludes the a priori estimates that are used to prove Theorem \ref{Th2.1}. The proof can be made formal by considering an approximating solution $u^{(n)}, n\in\mathbb{N}$, proving the above estimates for $u^{(n)}$, and then taking the limit as $n\to\infty$. We thus omit the details and refer the readers to \cite{KT} for further discussions.
\end{proof}

\bigskip
\noindent{\bf Acknowledgements.} The research of the second author was supported by NSF of China(11422106) and Fok Ying Tung Education Foundation (151001), the research of the first author and the last author is supported partially by
``The Fundamental Research Funds for Central Universities of China".

\end{document}